\documentclass[article]{amsart}
\usepackage{cases}
\usepackage{amsmath, amsfonts, pifont, amssymb}
\usepackage{verbatim}
\usepackage[bookmarks=true]{hyperref}

\usepackage{geometry}
\usepackage{enumerate}
\usepackage{color}

\usepackage{mathtools}
\mathtoolsset{showonlyrefs}

\usepackage{todonotes}

\usepackage{mathtools}

\newtheorem{theorem}{Theorem}[section]
\newtheorem{lemma}[theorem]{Lemma}
\newtheorem{proposition}[theorem]{Proposition}

\newcommand{\R}{\mathbb{R}}
\newcommand{\Z}{\mathbb{Z}}
\newcommand{\T}{\mathbb{T}}
\newcommand{\N}{\mathbb{N}}

\newcommand{\beq}{\begin{equation}}
\newcommand{\eeq}{\end{equation}}
\newcommand{\beqq}{\begin{equation*}}
\newcommand{\eeqq}{\end{equation*}}

\theoremstyle{definition}
\newtheorem{definition}[theorem]{Definition}

\theoremstyle{remark}
\newtheorem{remark}[theorem]{Remark}

\numberwithin{equation}{section}

\newcommand{\abs}[1]{\lvert#1\rvert}
\newcommand{\norm}[1]{\left\|#1\right\|}

\numberwithin{equation}{section}

\begin{document}

\address{Yongming Luo
\newline \indent Institut f\"ur Wissenschaftliches Rechnen, Technische Universit\"at Dresden\indent
\newline \indent  Zellescher Weg 25, 01069 Dresden, Germany.\indent }
\email{yongming.luo@tu-dresden.de}

\address{Xueying  Yu
\newline \indent Department of Mathematics, University of Washington\indent
\newline \indent  C138 Padelford Hall Box 354350, Seattle, WA 98195,\indent }
\email{xueyingy@uw.edu}

\address{Haitian Yue
\newline \indent Institute of Mathematical Sciences, ShanghaiTech University\newline\indent
Pudong, Shanghai, China.}
\email{yuehaitian@shanghaitech.edu.cn}

\address{Zehua Zhao
\newline \indent Department of Mathematics and Statistics, Beijing Institute of Technology, Beijing, China.
\newline \indent MIIT Key Laboratory of Mathematical Theory and Computation in Information Security, Beijing, China.}
\email{zzh@bit.edu.cn}

\title[On well-posedness results for CQNLS on $\mathbb{T}^3$]{On well-posedness results for the cubic-quintic NLS on $\mathbb{T}^3$}
\author{Yongming Luo, Xueying Yu, Haitian Yue and Zehua Zhao}


\begin{abstract}
We consider the periodic cubic-quintic nonlinear Schr\"odinger equation
\begin{align}\label{cqnls_abstract}
(i\partial_t +\Delta )u=\mu_1 |u|^2 u+\mu_2 |u|^4 u\tag{CQNLS}
\end{align}
on the three-dimensional torus $\mathbb{T}^3$ with $\mu_1,\mu_2\in \mathbb{R} \setminus\{0\}$. As a first result, we establish the small data well-posedness of \eqref{cqnls_abstract} for arbitrarily given $\mu_1$ and $\mu_2$. By adapting the crucial perturbation arguments in \cite{zhang2006cauchy} to the periodic setting, we also prove that \eqref{cqnls_abstract} is always globally well-posed in $H^1(\mathbb{T}^3)$ in the case $\mu_2>0$.
\end{abstract}

\maketitle

\setcounter{tocdepth}{1}
\tableofcontents


\parindent = 10pt
\parskip = 8pt

\section{Introduction and main results}
In this paper, we study the cubic-quintic nonlinear Schr\"odinger equation (CQNLS)
\begin{equation}\label{eq: main}
(i\partial_t+\Delta_{x})u=\mu_1|u|^2u+\mu_2|u|^4u
\end{equation}
on the three-dimensional torus $\T^3$, where $\mu_1,\mu_2\in\R\setminus\{0\}$ and $\T=\R/2\pi\Z$. The CQNLS \eqref{eq: main} arises in numerous physical applications such as nonlinear optics and Bose-Einstein condensate. Physically, the nonlinear potentials $|u|^2 u$ and $|u|^4 u$ model the two- and three-body interactions respectively and the positivity or negativity of $\mu_1$ and $\mu_2$ indicates whether the underlying nonlinear potential is repulsive (defocusing) or attractive (focusing). We refer to, for instance, \cite{phy2,phy3,phy1} and the references therein for a more comprehensive introduction on the physical background of the CQNLS \eqref{eq: main}. Mathematically, the CQNLS model \eqref{eq: main} on Euclidean spaces $\R^d$ ($d\leq 3$) has been intensively studied in \cite{carles2020soliton,Carles_Sparber_2021,Cheng2020,MurphyKillipVisanThreshold,killip_visan_soliton,Luo_JFA_2022,Murphy2021CPDE,SoaveSubcritical,SoaveCritical,tao2007nonlinear,zhang2006cauchy}, where well-posedness and long time behavior results for solutions of \eqref{eq: main} as well as results for existence and (in-)stability of soliton solutions of \eqref{eq: main} were well established.

In this paper, we aim to give some first well-posedness results for \eqref{eq: main} in the periodic setting, which, to the best of our knowledge, have not existed to that date. We also restrict ourselves to the most appealing case $d=3$, where the quintic potential is energy-critical. (By `energy-critical', we mean the energy of solution is invariant under the scaling variance. See \cite{Iteam1} for more details.) In this case, the well-posedness of \eqref{eq: main} shall also depend on the profile of the initial data and the analysis becomes more delicate and challenging.

Our first result deals with the small data well-posedness of \eqref{eq: main}, which is given in terms of the function spaces $Z'(I),X^1(I)$ defined in Section \ref{sec Pre} for a given time slot $I$.

\begin{theorem}[Small data well-posedness]\label{thm:lwp}
Consider \eqref{eq: main} on a time slot $I=(-T,T)\subset\R$ with some $T\in(0,\infty)$. Let $u_0\in H^1(\T^3)$ satisfies
$\|u_0\|_{ H^1(\mathbb{T}^3) }\leq E$
for some $E>0$. Then there exists $\delta=\delta(E,T)>0$ such that if
\begin{align}\label{smallness assumption}
\|e^{it\Delta}u_0\|_{Z'(I)}\leq \delta,
\end{align}
then \eqref{eq: main} admits a unique strong solution $u\in X^{1}(I)$ with initial data $u(0,x)=u_0(x)$.
\end{theorem}

The proof of Theorem \ref{thm:lwp} is based on a standard application of the contraction principle. Nonetheless, one of the major challenges in proving well-posedness of dispersive equations on tori is the rather exotic Strichartz estimates, leading in most cases to very technical and cumbersome proofs. In the energy-subcritical setting, Strichartz estimates for periodic nonlinear Schr\"odinger equations (NLS) were first established by Bourgain \cite{Bourgain1} by appealing  to the number-theoretical methods. In our case, where an energy-critical potential is present, we shall make use of the Strichartz estimates introduced by Herr-Tataru-Tzvetkov \cite{HTT1} based on the atomic space theory, which in turn initiates applications of the function spaces defined in Section \ref{sec Pre}. Notice also that in comparison to the purely quintic NLS model studied in \cite{HTT1}, an additional cubic term should also be dealt in our case. A new bilinear estimate on $\T^3$ will therefore be proved in order to obtain a proper estimate for the cubic potential, and we refer to Lemma \ref{lem:bilinear} for details. For interested readers, we also refer to \cite{CGZ,CZZ,HTT1,HTT2,IPT3,IPRT3,yang2018global,yang2023scattering,YYZ,Z1,Z2,ZhaoZheng} for further well-posedness results for NLS (with single nonlinear potential) on tori or waveguide manifolds based on the atomic space theory.  (See \cite{sire2022scattering,yu2021global} for other dispersive equations on waveguides.)

Despite that small data well-posedness results are satisfactory to certain extent, it is more interesting (and hence also more challenging) to deduce well-posedness results where the initial data are not necessarily small. We focus here on the particular scenario where the quintic potential is repulsive ($\mu_2>0$), which is motivated by the following physical concern: Consider for instance the focusing cubic NLS{\footnote{When $d=1$, the mass-subcritical nature of the nonlinear potential, combining with conservation of mass and energy, guarantees the global well-posedness of \eqref{cubic nls} in $H^1(\R)$ as well as $H^1(\T)$.}}
\begin{align}\label{cubic nls}
(i\partial_t +\Delta)u=-|u|^2 u
\end{align}
on $\R^d$ with $d\in\{2,3\}$. By invoking the celebrated Glassey's identity \cite{Glassey1977} one may construct finite time blow-up solutions of \eqref{cubic nls} for initial data lying in weighted $L^2$-spaces or satisfying radial symmetric conditions, see for instance \cite{Cbook} for a proof. Surprisingly, in contrast to the rigorously derived blow-up results, collapse of the wave function does not appear in many actual experiments. It is therefore suggested to incorporate a higher order repulsive potential into \eqref{cubic nls}, the case that the repulsive potential is taken as the three-body interaction leads to the study of CQNLS \eqref{eq: main}. More interestingly, it turns out that in the presence of a quintic stabilizing potential, \eqref{eq: main} is in fact globally well-posed for arbitrary initial data in $H^1(\R^d)$. While for $d=2$ this follows already from conservation laws and the energy-subcritical nature of \eqref{eq: main} on $\R^2$, the proof in the case $d=3$, where the quintic potential becomes energy-critical, is more involved. A rigorously mathematical proof for confirming such heuristics in $d=3$ was first given by Zhang \cite{zhang2006cauchy}. The idea from \cite{zhang2006cauchy} can be summarized as follows: We may consider \eqref{eq: main} as a perturbation of the three dimensional defocusing quintic NLS
\begin{align}\label{quintic nls}
(i\partial_t +\Delta)u=|u|^4 u
\end{align}
whose global well-posedness in $\dot{H}^1(\R^3)$ was shown in \cite{Iteam1}. We then partition the time slot $I$ into disjoint adjacent small intervals $I=\cup_{j=0}^m\, I_j$. On each of these intervals, the cubic term is expected to be ``small'' because of the smallness of the subinterval, and by invoking a stability result we may prove that \eqref{eq: main} is well-posed on a given $I_j$. Based on the well-posedness result on $I_j$ we are then able to prove the same result for the consecutive interval $I_{j+1}$ and so on. Starting from the interval $I_0$ and repeating the previous procedure inductively over all $I_j$ follows then the desired claim.

Inspired by the result given in \cite{zhang2006cauchy}, we aim to prove the following analogous global well-posedness result for \eqref{eq: main} on $\T^3$ in the case $\mu_2>0$.

\begin{theorem}[Global well-posedness in the case $\mu_2>0$]\label{thm:gwp}
Assume that $\mu_2>0$. Then \eqref{eq: main} is globally well-posed in $H^1(\T^3)$ in the sense that for any $T>0$ and $u_0\in H^1(\T^3)$, \eqref{eq: main} possesses a solution $u\in X^1(I)$ on $I=(-T,T)$ with $u(0)=u_0$.
\end{theorem}

\begin{remark}
We note that one can also obtain the waveguide analogues of Theorem \ref{thm:gwp}, (i.e. considering \eqref{eq: main} posed on $\mathbb{R}^2 \times \mathbb{T}$ and $\mathbb{R} \times \mathbb{T}^2$) with suitable modifications. Moreover, for the $\mathbb{R}^2 \times \mathbb{T}$ case, scattering behavior is also expected according to existing literature (see \cite{Z2}). However, the scattering result require a lot more than this GWP scheme and we leave it for future considerations. 
\end{remark}

\begin{remark}
It is worth mentioning that  the same global well-posedness result for the supercubic-quintic NLS
\[
(i\partial_t +\Delta )u=\mu_1 |u|^{p-1} u+\mu_2 |u|^4 u, \quad\text{for}\quad 3<p<5,
\]
 is expected to be yielded by adapting the nonlinear estimates in Section \ref{sec LWP} into the fractional product case (see \cite{GyuEunLee2019} for reference, see also \cite{zhang2006cauchy} for the Euclidean case). We leave it for interested readers.
\end{remark}

We follow closely the same lines from \cite{zhang2006cauchy} to prove Theorem \ref{thm:gwp}. In comparison to the Euclidean case, there are essentially two main new ingredients needed for the proof of Theorem \ref{thm:gwp}:
\begin{enumerate}[(i)]
\item
The Black-Box-Theory from \cite{Iteam1} is replaced by the one from \cite{IPT3} for \eqref{quintic nls} on $\T^3$.
\item
The estimates are correspondingly modified (in a very technical and subtle way) in order to apply the Strichartz estimates based on the atomic space theory.
\end{enumerate}
We refer to Section \ref{sec GWP} for  the proof of Theorem \ref{thm:gwp} in detail. For further applications of such interesting perturbation arguments on NLS with combined powers, we also refer to \cite{tao2007nonlinear}.

\begin{remark}
By a straightforward scaling argument it is not hard to see that both Theorems \ref{thm:lwp} and \ref{thm:gwp} extend verbatim to the case where $\T^3$ is replaced by any rational torus. Such direct scaling argument, however, does not apply to irrational tori. Nevertheless, thanks to the ground breaking work of Bourgain and Demeter \cite{BD} we also know that the Strichartz estimates established in \cite{HTT1} are in fact  available for irrational tori, by which we are thus able to conclude that Theorems \ref{thm:lwp} and \ref{thm:gwp} indeed remain valid for arbitrary tori regardless of their rationality. For simplicity we will keep working with the torus $\T^3$ in the rest of the paper.
\end{remark}

We outline the structure of the rest of the paper. In Section \ref{sec Pre}, we summarize the notations and definitions which will be used throughout the paper and define the function spaces applied in the Cauchy problem \eqref{eq: main}. In Sections \ref{sec LWP} and \ref{sec GWP}, we prove Theorems \ref{thm:lwp} and \ref{thm:gwp} respectively.


\subsection*{Acknowledgment} Y. Luo was funded by Deutsche Forschungsgemeinschaft (DFG) through the Priority Programme SPP-1886 (No. NE 21382-1). H. Yue was supported by the Shanghai Technology Innovation Action Plan (No. 22JC1402400), a Chinese overseas high-level young talents program (2022) and the start-up funding of ShanghaiTech University. Z. Zhao was supported by the NSF grant of China (No. 12101046, 12271032), Chinese overseas high-level young talents program (2022) and the Beijing Institute of Technology Research Fund Program for Young Scholars.

\section{Preliminaries}\label{sec Pre}
In this section, we first discuss notations used in the rest of the paper, introduce the function spaces with their properties that we will be working on, and list some useful tools from harmonic analysis.
\subsection{Notations}
We use the notation $A\lesssim B$ whenever there exists some positive constant $C$ such that $A\leq CB$. Similarly we define $A\gtrsim B$ and we use $A\sim B$ when $A\lesssim B\lesssim A$. For simplicity, we hide in most cases the dependence of the function spaces on their spatial domain in their indices. For example $L_x^2=L^2(\T^3)$, $\ell_k^2=\ell^2(\Z^3)$
and so on. However, when the space is involved with time we still display the underlying temporal interval such as $L_{t,x}^p(I)$, $L_t^pL_x^q(I)$, $L_t^\infty \ell_k^2(\R)$ etc. We also frequently write $\|\cdot\|_p:=\|\cdot\|_{L_{x}^p}$.

\subsection{Fourier transforms and Littlewood-Paley projections}
Throughout the paper we use the following Fourier transformation on $\mathbb{T}^3$:
\begin{equation*}
    (\mathcal{F} f)(\xi)=\widehat{f}(\xi)=(2\pi)^{-\frac32}\int_{\T^d}f(x)e^{-ix\cdot \xi} \, dx
\end{equation*}
for $\xi\in \mathbb{Z}^3$. The corresponding Fourier inversion formula is then given by
\begin{equation*}
f(x)=(2\pi)^{-\frac32} \sum_{\xi\in \mathbb{Z}^3}  (\mathcal{F} f)(\xi)e^{ix\cdot \xi}.
\end{equation*}
By definition, the Schr{\"o}dinger propagator $e^{it\Delta}$ is defined by
\begin{equation*}
    \left(\mathcal{F} e^{it\Delta}f\right)(\xi)=e^{-it|\xi|^2}(\mathcal{F} f)(\xi).
\end{equation*}

Next we define the Littlewood-Paley projectors. We fix some even decreasing function $\eta\in C_c^\infty(\R;[0,1])$ satisfying $\eta(t)\equiv 1$ for $|t|\leq 1$ and $\eta(t)\equiv 0$ for $|t|\geq 2$. For a dyadic number $N\geq 1$ define $\eta_N:\Z^3\to [0,1]$ by
\begin{gather*}
\eta_N(\xi)=\eta(|\xi|/N)-\eta(2|\xi|/N),\quad N\geq 2,\\
\eta_N(\xi)=\eta(|\xi|),\quad N=1.
\end{gather*}
Then the Littlewood-Paley projector $P_N$ ($N\geq 1$) is defined as the Fourier multiplier with symbol $\eta_N$. For any $N\in (0,\infty)$, we also define
\begin{equation}
P_{\leq N}:=\sum_{M\in 2^{\N},M\leq 
N}P_M,\quad P_{> N}:=\sum_{M\in 2^{\N},M>N}P_M.
\end{equation}

\subsection{Strichartz estimates}
As already pointed out in the introductory section, unlike the Euclidean case, the Strichartz estimates on (rational or irrational) tori are generally proved in a highly non-trivial way and in most cases only frequency-localized estimates can be deduced. For our purpose we will make use of the following Strichartz estimate proved by Bourgain and Demeter \cite{BD} (see also \cite{Bourgain1,KV1}).

\begin{proposition}[Frequency-localized Strichartz estimates on  $\T^3$, \cite{BD}]
Consider the linear Schr\"odinger propagator $e^{it\Delta}$ on a (rational or irrational) three-dimensional torus. Then for $p>\frac{10}{3}$ we have for any time slot $I$ with $|I|\leq 1$
\begin{equation}\label{strichartz}
\|e^{it\Delta} P_N  f\|_{L^p_{t,x}(I \times \T^3)} \lesssim_p N^{\frac{3}{2}-\frac{5}{p}} \|P_N  f\|_{L^2_x (\T^3)}.
\end{equation}
\end{proposition}

\subsection{Function spaces}
Next, we define the function spaces and collect some of their useful properties which will be used for the Cauchy problem \eqref{eq: main}. We begin with the definitions of $U^p$- and $V^p$-spaces introduced in \cite{HadacHerrKoch2009}.

\begin{definition}[$U^p$-spaces]\label{def up}
Let $1\leq p < \infty$, $\mathcal{H}$ be a complex Hilbert space and $\mathcal{Z}$ be the set of all finite partitions $-\infty<t_0<t_1<...<t_K\leq \infty$ of the real line. A $U^p$-atom is a piecewise constant function $a:\mathbb{R} \rightarrow \mathcal{H}$ defined by
\begin{align*}
a=\sum_{k=1}^{K}\chi_{[t_{k-1},t_k)}\phi_{k-1},
\end{align*}
where $\{t_k\}_{k=0}^{K} \in \mathcal{Z}$ and $\{\phi_k\}_{k=0}^{K-1} \subset \mathcal{H}$ with $\sum_{k=0}^{K}\|\phi_k\|^p_{\mathcal{H}}=1$. The space $U^p(\mathbb{R};\mathcal{H})$ is then defined as the space of all functions $u:\mathbb{R}\rightarrow \mathcal{H}$ such that $u=\sum_{j=1}^{\infty}\lambda_j a_j$ with $U^p$-atoms $a_j$ and $\{\lambda_j\} \in \ell^1$. We also equip the space $U^p(\mathbb{R};\mathcal{H})$ with the norm
\begin{align*}
\|u\|_{U^p}:=\inf\{\sum^{\infty}_{j=1}|\lambda_j|:u=\sum_{j=1}^{\infty}\lambda_j a_j,\,\lambda_j\in \mathbb{C},\, a_j \text{ are } U^p\textmd{-atoms}\}.
\end{align*}
\end{definition}

\begin{definition}[$V^p$-spaces]
We define the space $V^p(\mathbb{R},\mathcal{H})$ as the space of all functions $v:\mathbb{R} \rightarrow \mathcal{H}$ such that
\begin{align*}
\|v\|_{V^p}:=\sup\limits_{\{t_k\}^K_{k=0} \in \mathcal{Z}}(\sum_{k=1}^{K}\|v(t_k)-v(t_{k-1})\|^p_{\mathcal{H}})^{\frac{1}{p}} < +\infty,
\end{align*}
where we use the convention $v(\infty)=0$. Also, we denote by $V^p_{rc}(\mathbb{R},\mathcal{H})$ the closed subspace of $V^p(\mathbb{R},\mathcal{H})$ containing all right-continuous functions $v$ with $\lim\limits_{t\rightarrow -\infty}v(t)=0$.
\end{definition}
In our context we shall set the Hilbert space $\mathcal{H}$ to be the Sobolev space $H_x^s$ with $s\in\R$, which will be the case in the remaining parts of the paper.
\begin{definition}[$U_{\Delta}^p$- and $V_{\Delta}^p$-spaces in \cite{HadacHerrKoch2009}]
For $s\in \mathbb{R}$ we let $U^p_{\Delta}H_x^s(\R)$ resp. $V^p_{\Delta}H_x^s(\R)$ be the spaces of all functions such that $e^{-it\Delta}u(t)$ is in $U^p(\mathbb{R},H_x^s)$ resp. $V^p_{rc}(\mathbb{R},H_x^s)$, with norms
\begin{align*}
\|u\|_{U^p_{\Delta}H_x^s(\R)}=\|e^{-it\Delta}u\|_{U^p(\mathbb{R},H_x^s)}, \quad \|u\|_{V^p_{\Delta}H_x^s(\R)}=\|e^{-it\Delta}u\|_{V^p(\mathbb{R},H_x^s)}.
\end{align*}
\end{definition}

Having defined the $U_\Delta^p$- and $V_\Delta^p$-spaces we are now ready to formulate the function spaces for studying the Cauchy problem \eqref{eq: main}. For $C=[-\frac{1}{2},\frac{1}{2})^3 \in \mathbb{R}^3$ and $z\in \mathbb{R}^3$ let $C_z=z+C$ be the translated unit cube centered at $z$ and define the sharp projection operator $P_{C_z}$ by
\begin{align*}
\mathcal{F}(P_{C_z} f)(\xi)=\chi_{C_z}(\xi) \mathcal{F} (f)  (\xi),\quad\xi\in\Z^3,
\end{align*}
where $\chi_{C_z}$ is the characteristic function restrained on $C_z$. We then define the $X^s$- and $Y^s$-spaces as follows:

\begin{definition}[$X^s$- and $Y^s$-spaces]
For $s\in \mathbb{R}$ we define the $X^s(\R)$- and $Y^s(\R)$-spaces through the norms
\begin{align*}
\|u\|_{X^s(\mathbb{R})}^2&:=\sum_{z\in \mathbb{Z}^3} \langle z \rangle^{2s} \|P_{C_z} u\|_{U_{\Delta}^2(\mathbb{R};L_x^2)}^2,\\
\|u\|_{Y^s(\mathbb{R})}^2&:=\sum_{z\in \mathbb{Z}^3} \langle z \rangle^{2s} \|P_{C_z} u\|_{V_{\Delta}^2(\mathbb{R};L_x^2)}^2 .
\end{align*}
\end{definition}

For an interval $I \subset \mathbb{R}$ we also consider the restriction spaces $X^s(I),Y^s(I)$ etc. For these spaces we have the following useful embedding:

\begin{proposition}[Embedding between the function spaces, \cite{HadacHerrKoch2009}]
For $2< p< q<\infty$ we have
\begin{align*}
U^2_{\Delta}H_x^s \hookrightarrow X^s\hookrightarrow  Y^s \hookrightarrow V^2_{\Delta}H_x^s\hookrightarrow  U^p_{\Delta}H_x^s\hookrightarrow U^q_{\Delta}H_x^s \hookrightarrow L^{\infty}H_x^s.
\end{align*}
\end{proposition}

As usual, the proofs of the well-posed results rely on the contraction principle and thus a dual norm estimation for the Duhamel term is needed. In the periodic setting, the dual norm is given as the $N^s$-norm, which is defined as follows:

\begin{definition}[$N^s$-norm]
On a time slot $I$ we define the $N^s(I)$-norm for $s\in\R$ by
\begin{equation*}
\| h\|_{N^s(I)}=\|\int_{a}^{t} e^{i(t-s)\Delta} h(s) \, ds \|_{X^s(I)} .
\end{equation*}
\end{definition}

The following proposition sheds light on the duality of the spaces $N^1(I)$ and $Y^{-1}(I)$.

\begin{proposition}[Duality of $N^1(I)$ and $Y^{-1}(I)$ in \cite{HTT1}]\label{prop:dual}
The spaces $N^1(I)$ and $Y^{-1}(I)$ satisfy the following duality inequality
\begin{align*}
\|f\|_{N^1(I)} \lesssim \sup_{\|v\|_{Y^{-1}(I)}\leq 1} \int_{I \times  \mathbb{T}^3} f(t,x)\overline{v(t,x)} \, dxdt.
\end{align*}
Moreover, the following estimate holds for any smooth ($H_x^1$-valued) function $g$ on an interval $I=[a,b]$:
\begin{align*}
\|g\|_{X^1(I)}\lesssim \|g(a)\|_{H_x^1}+(\sum_N \|P_N(i\partial_t+\Delta)g\|^2_{L_t^1 H_x^1(I)})^{\frac{1}{2}}.
\end{align*}
\end{proposition}

For our purpose we shall also need appeal to the $Z$-norm which is defined as follows:
\begin{definition}[$Z$-norm]
For a time slot $I$ we define the $Z(I)$-norm by
\begin{align*}
\|u\|_{Z(I)}=\sup\limits_{J\subset I,|J|\leq 1}(\sum_{N\geq 1}N^{3} \| P_N u\|_{L_{t, x}^4(J)}^4)^\frac{1}{4}.
\end{align*}
\end{definition}
As a direct consequence of the Strichartz estimates it is easy to verify that
\begin{align*}
\|u\|_{Z(I)} \lesssim \|u\|_{X^1(I)}.
\end{align*}
For those readers who are familiar with NLS on the standard Euclidean space $\R^d$, we also note that intuitively, the $X^1$- and $Z$-norms play exactly the same roles as the norm $\|\cdot\|_{S^1}:=\|\cdot\|_{L_t^\infty H_x^1\cap L_t^2 W_x^{1,6}}$ and $L^{10}_{t,x}$-norm for the quintic NLS on $\R^3$ respectively. Nevertheless, the $Z$-norm can not be directly applied to prove the well-posedness results. To that end, we introduce the $Z'$-norm defined by
$$ \|u\|_{Z'}:=\|u\|_{Z}^{\frac12}\|u\|_{X^1}^{\frac12}$$
which will be more useful for the proof of Theorem \ref{thm:lwp}.

\subsection{Conservation laws}
We end this section by introducing the mass $M(u)$ and energy $E(u)$ associating to the NLS flow \eqref{eq: main}:
\begin{align}\label{eq ME}
\begin{aligned}
M(u) &=\int_{\T^3}|u|^2 \, dx,\\
E(u) & =\int_{\T^3}\frac{1}{2}|\nabla u|^2+\frac{\mu_1}{4}|u|^4+\frac{\mu_2}{6}|u|^6  \,dx. 
\end{aligned}
\end{align}
It is well-known that both mass and energy are conserved over time along the NLS flow \eqref{eq: main}.

As a direct application of conservation laws and H\"older's inequality, we have the following uniform estimate of the kinetic energy $\|\nabla u\|_{L_t^\infty L_x^2(I \times \T^3)}^2 = : \|\nabla u\|_{L_t^\infty L_x^2(I) }^2 $ for a solution $u$ of \eqref{eq: main}. (As mentioned in Notations, we omit the space $\T^3$ for convenience.) We include the proof below for completeness (see the original argument in \cite[Sec. 2.2]{zhang2006cauchy}). 

\begin{lemma}\label{zhang uniform kinetic}
Let $u\in X^1(I)$ be a solution of \eqref{eq: main} with $u(0)=u_0$. Then 
\begin{align}
\|\nabla u\|_{L_t^\infty L_x^2(I)}^2\lesssim E(u_0)+M(u_0)^2 .
\end{align}
\end{lemma}

\begin{proof}[Proof of Lemma \ref{zhang uniform kinetic}]
Recall the mass and energy defined in \eqref{eq ME}. If both $\mu_1$ and $\mu_2$ are positive, it is  easy to see that for any $t$
\begin{align}
\norm{\nabla u(t)}_{L_x^2}^2 \lesssim E .
\end{align}
If $\mu_1 <0$ and $\mu_2 >0$, then we use the following inequality for some $C(\mu_1 , \mu_2)$
\begin{align}
-\frac{\abs{\mu_1}}{4} \abs{u(t,x)}^4 + \frac{\abs{\mu_2}}{6} \abs{u(t,x)}^6 \geq - C(\mu_1 , \mu_2) \abs{u(t,x)}^2
\end{align}
to conclude that for any $t$
\begin{align}
\norm{\nabla u(t)}_{L_x^2}^2 \lesssim E + M^2 .
\end{align}
\end{proof}

\section{Proof of Theorem \ref{thm:lwp}}\label{sec LWP}
In this section we give the proof of Theorem \ref{thm:lwp}. As the precise value of $|I|=2T$ has only impact on the numerical constants, without loss of generality, we may also assume that $|I|\leq 1$ throughout this section.

We begin with recording a trilinear estimate deduced in \cite{IPT3}.

\begin{lemma}[Trilinear estimate, \cite{IPT3}]\label{lem Trilinear}
Suppose that $ u_i=P_{N_i}u$, for $i=1,2,3$ satisfying $N_1\geq N_2 \geq N_3\geq 1$. Then there exists some $\delta>0$ such that
\begin{equation}\label{equation}
\|u_1 u_2 u_3\|_{L_{t,x}^2(I)} \lesssim \left(\frac{N_3}{N_1}+\frac{1}{N_2}\right)^\delta \|u_1\|_{Y^0(I)} \|u_2\|_{Z'(I)} \|u_3\|_{Z'(I)} .
\end{equation}
\end{lemma}

For dealing with the cubic term, we also need the following bilinear estimate.

\begin{lemma}[Bilinear estimate]\label{lem:bilinear}
Suppose that $u_i=P_{N_i}u$, for $i=1,2$ satisfying $N_1\geq N_2\geq 1$. Then there exists some $\kappa>0$ such that
\begin{equation}\label{est:bilinear}
    \|u_1u_2\|_{L^2_{t,x}(I)}\lesssim \left(\frac{N_2}{N_1}+\frac{1}{N_2}\right)^\kappa|I|^{\frac{1}{20}}\|u_1\|_{Y^0(I)} \|u_2\|_{Z'(I)}.
\end{equation}
\end{lemma}
\begin{proof}[Proof of Lemma \ref{lem:bilinear}]
For any cube $C$ centered at $\xi\in \mathbb{Z}^3$ of size $N_2$, using H\"older's inequality and the Strichartz estimate \eqref{strichartz}, we have
\[
\begin{aligned}
\|(P_C u_1) u_2\|_{L^2_{t,x}(I)}&\lesssim \|P_C u_1\|_{L^4_{t,x}(I)}  \|u_2\|_{L^4_{t,x}(I)}\lesssim |I|^\frac{1}{10} \|P_C u_1\|_{L^{\frac{20}{3}}_{t,x}(I)}  \|u_2\|_{L^4_{t,x}(I)}\\
&\lesssim |I|^\frac{1}{10} \|P_C u_1\|_{U_\Delta^{\frac{20}{3}}L_x^2(I)}  \left(N_2^{\frac{3}{4}}\|u_2\|_{L^4_{t,x}(I)}\right)\lesssim
|I|^\frac{1}{10} \|P_C u_1\|_{Y^0(I)}  \left(N_2^{\frac{3}{4}}\|u_2\|_{L^4_{t,x}(I)}\right).
\end{aligned}\]
Using the orthogonality and summability properties of $Y^0(I)$ and the definition of $Z(I)$, the above estimate provides
\begin{equation}\label{eq:dagger}
    \begin{aligned}
    \|u_1 u_2\|^2_{L^2_{t,x}(I)}&\lesssim |I|^\frac{1}{5} \sum_{C} \|P_{C} u_1\|^2_{Y^0(I)}  \left(N_2^{\frac{3}{4}}\|u_2\|_{L^4_{t,x}(I)}\right)^2\lesssim |I|^\frac{1}{5} \|u_1\|^2_{Y^0(I)}  \|u_2\|^2_{Z(I)},
    \end{aligned}
\end{equation}
where the sum is over all $\xi\in N_2^{-1}\Z^3$.
It remains to prove
\begin{equation}\label{eq:ddagger}
    \begin{aligned}
   \|u_1u_2\|_{L^2_{t,x}(I)}\lesssim \left(\frac{N_2}{N_1}+\frac{1}{N_2}\right)^{\kappa_0} \|u_1\|_{Y^0(I)} \|u_2\|_{Y^1(I)}
    \end{aligned}
\end{equation}
for some $\kappa_0>0$, the desired claim follows then from interpolating \eqref{eq:dagger} and \eqref{eq:ddagger} and the embedding $X^1\hookrightarrow Y^1$. Again, using the orthogonality and summability properties of $Y^0(I)$ and Strichartz estimate \eqref{strichartz}, we obtain that
\[
\begin{aligned}
   \|u_1 u_2\|^2_{L^2_{t,x}(I)} &\lesssim \sum_C \|(P_C u_1) u_2\|^2_{L^2_{t,x}(I)}\lesssim \sum_C \left( N_2^{\frac{1}{2}} \|P_C u_1\|_{U^4_\Delta L_x^2(I)}  \|u_2\|_{U^4_\Delta L_x^2(I)} \right)^2\\
   &\lesssim \sum_{C} \left(
  N_2^{-\frac{1}{2}}  \|P_C u_1\|_{Y^0(I)} \|u_2\|_{Y^1(I)} \right)^2\lesssim N_2^{-1}  \|u_1\|^2_{Y^0(I)}  \|u_2\|^2_{Y^1(I)},
\end{aligned}
\]
as desired.
\end{proof}

As a direct consequence of the multilinear estimates deduced from Lemmas \ref{lem Trilinear} and \ref{lem:bilinear}, we immediately obtain the following nonlinear estimates.
\begin{lemma}[Nonlinear estimates]\label{lem Nonlinear}
For $u_k \in X^1(I)$, $k=1,2,3,4,5$, the following estimates
\begin{align}
\Big\|\prod_{i=1}^{5} \widetilde{u}_i\Big\|_{N^1(I)}&\lesssim \sum_{\{i_1,...i_5\}=\{1,2,3,4,5\} }\|u_{i_1}\|_{X^1(I)} \cdot \prod_{i_k \neq i_1} \|u_{i_k}\|_{Z'(I)},\label{est:nonlinear1}\\
\Big\|\prod_{i=1}^{3} \widetilde{u}_i\Big\|_{N^1(I)}&\lesssim|I|^{\frac{1}{20}} \sum_{\{ i_1,i_2,i_3 \}=\{ 1,2,3 \} }\|u_{i_1}\|_{X^1(I)} \cdot \prod_{i_k \neq i_1} \|u_{i_k}\|_{Z'(I)}\label{est:nonlinear2}
\end{align}
hold true, where $\widetilde{u} \in \{ u, \bar{u} \}$.
\end{lemma}
\begin{proof}[Proof of  Lemma \ref{lem Nonlinear}]
\eqref{est:nonlinear1} and \eqref{est:nonlinear2} can be proved, words by words, by using the arguments from \cite[Lem. 3.2]{IPT3} and \cite[Lem. 3.2]{IPRT3}, respectively, which make use of Lemma \ref{lem Trilinear} as well as Lemma \ref{lem:bilinear}. We thus omit the repeating arguments.
\end{proof}

Having all the preliminaries we are in a position to prove Theorem \ref{thm:lwp}.

\begin{proof}[Proof of Theorem \ref{thm:lwp}]
We define the contraction mapping
\begin{align*}
\Phi(u)&:= e^{it\Delta}u_0-i\mu_1\int_0^t e^{i(t-s)\Delta}|u|^2 u \, ds
-i\mu_2\int_0^t e^{i(t-s)\Delta}|u|^4 u \, ds.
\end{align*}
We aim to show that by choosing $\delta_0$ sufficiently small, the mapping $\Phi$ defines a contraction on the metric space
$$ \mathcal{S}:=\{u\in X^1(I): \|u\|_{X^1(I)}\leq 2CE,\,\|u\|_{Z'(I)}\leq 2\delta\},$$
where $C\geq 1$ is some universal constant. The space $\mathcal{S}$ is particularly a complete metric space equipping with the metric $\rho(u,v)
:=\|u-v\|_{X^1(I)}$. First we show that for $\delta$ small we have $\Phi(\mathcal{S})\subset \mathcal{S}$. Indeed, using Lemma \ref{lem Nonlinear} we obtain
\begin{align*}
\|\Phi(u)\|_{X^1(I)}&\leq \|e^{it\Delta}u_0\|_{X^1(I)}+C\|u\|_{X^1(I)}\|u\|^2_{Z'(I)}+C\|u\|_{X^1(I)}\|u\|^4_{Z'(I)}\nonumber\\
&\leq C\|u_0\|_{H_x^1}+C(2CE)(2C\delta)^2+C(2CE)(2C\delta)^4\nonumber\\
&\leq CE(1+(2C)^3\delta^2+(2C)^5\delta^4)\leq 2CE,\\
\nonumber\\
\|\Phi(u)\|_{Z'(I)}&\leq \|e^{it\Delta}u_0\|_{Z'(I)}+C\|u\|_{X^1(I)}\|u\|^2_{Z'(I)}+C\|u\|_{X^1(I)}\|u\|^4_{Z'(I)}\nonumber\\
&\leq \delta+C(2CE)(2C\delta)^2+C(2CE)(2C\delta)^4\leq 2\delta
\end{align*}
by choosing $\delta$ sufficiently small. 

It is left to show that $\Phi$ is a contraction for small $\delta$. Again, using Lemma \ref{lem Nonlinear} we obtain
\begin{align*}
\|\Phi(u)-\Phi(v)\|_{X^1(I)}&\leq C(\|u\|_{X^1(I)}+\|v\|_{X^1(J)})(\|u\|_{Z'(I)}+\|v\|_{Z'(I)})\|u-v\|_{X^1(I)}\nonumber\\
& \quad +C(\|u\|_{X^1(I)}+\|v\|_{X^1(J)})(\|u\|_{Z'(I)}+\|v\|_{Z'(I)})^3\|u-v\|_{X^1(I)}\nonumber\\
&\leq C(4CE)(4C\delta+(4C\delta)^3)\|u-v\|_{X^1(I)}\leq\frac{1}{2}\|u-v\|_{X^1(I)}
\end{align*}
by choosing $\delta$ small. This completes the proof of Theorem \ref{thm:lwp}.
\end{proof}

\section{Proof of Theorem \ref{thm:gwp}}\label{sec GWP}
In this section we prove Theorem \ref{thm:gwp}. Again, without loss of generality, we may assume that $|I|\leq 1$ and $\mu_2=1$. The goal is therefore to show that \eqref{eq: main} is well-posed on $I$ without imposing the smallness condition \eqref{smallness assumption}. We firstly introduce the following large data Black-Box-Theory for defocusing quintic NLS on $\T^3$ from \cite{IPT3}.
\begin{theorem}[GWP of the defocusing quintic NLS on $\T^3$, \cite{IPT3}]\label{thm:black box}
Consider the defocusing quintic NLS
\begin{align}\label{quintic nls t3}
(i\partial_t+\Delta)v=|v|^4 v
\end{align}
on $I=(-T,T)$ with $|I|\leq 1$. Then for any $v_0\in H_x^1$, \eqref{quintic nls t3} possesses a unique solution $v\in X^1(I)$ with $v(0)=v_0$. Moreover, we have
\begin{align}\label{global control quintic}
\|v\|_{X^1(I)}+\|v\|_{Z(I)}\leq C(M(v_0),E(v_0))<\infty.
\end{align}
\end{theorem}

We are now ready to prove Theorem \ref{thm:gwp}.

\begin{proof}[Proof of Theorem \ref{thm:gwp}]
Consider first a subinterval $J=(a,b)\subset I$ and the difference NLS equation
\begin{align}\label{diff eq}
(i\partial_t +\Delta)w=\mu_1 |v+w|^2(v+w)+|v+w|^4(v+w)-|v|^4v
\end{align}
on $J$ with $w(a)=0$ and $v$ a solution of \eqref{quintic nls t3} with $v(a)=u(a)$. The proof of Theorem \ref{thm:gwp} for the interval $J$ follows once we are able to prove that \eqref{diff eq} possesses a unique solution $w\in X^1(J)$. By \eqref{global control quintic} and the definition of the $Z'$-norm, we may partition $I$ into disjoint consecutive intervals $I=\cup_{j=0}^m \,I_j$ with $I_j=[t_j,t_{j+1}]$ such that
\begin{align}
\|v\|_{Z'(I_j)}\leq \eta
\end{align}
for some to be determined small $\eta$. From now on we consider those $I_j$ such that $I_j\cap J\neq\varnothing$ (say $m_1\leq j\leq m_2$). Without loss of generality we may also assume that $J=\cup_{j=m_1}^{m_2}\,I_j$. Suppose at the moment that for a given $I_j$, the solution $w$ satisfies
\begin{align}\label{wj}
\max\{\|w\|_{L_t^\infty H_x^1(I_j)},\|w\|_{X^1(I_j)}\}\leq (2C)^j |J|^{\frac{1}{20}}
\end{align}
with some universal constant $C>0$.

We consider the contraction mapping
$$ \Gamma_j w:=e^{i(t-t_j)\Delta}w(t_j)-i\int_{t_j}^{t}e^{i(t-s)\Delta}(\mu_1|v+w|^2(v+w)+|v+w|^4(v+w)-|v|^4v)(s)\,ds$$
on the set
$$\mathcal{S}_j:=\{w\in X^1(I_j):\max\{\|w\|_{L_t^\infty H_x^1(I_j)},\|w\|_{X^1(I_j)}\}\leq (2C)^j |J|^{\frac{1}{20}}\},$$
which is a complete metric space with respect to the metric
\begin{align}
\rho(u_1,u_2):=\|u_1-u_2\|_{X^1(I_j)} .
\end{align}
We show that by choosing $\eta$ and $|J|$ small, the mapping $\Gamma_j$ defines a contraction on $\mathcal{S}_j$. Using Strichartz estimates, Lemma \ref{lem Nonlinear}, the embedding $X^1\hookrightarrow Z'$ and the inductive hypothesis
\begin{align}
\|w(t_j)\|_{H_x^1}\leq (2C)^{j-1}|J|^{\frac{1}{20}}
\end{align}
we obtain
\begin{align}
& \quad \max\{\|\Gamma_j w\|_{L_t^\infty H_x^1(I_j)},\|\Gamma_j w\|_{X^1(I_j)}\}\nonumber\\
& \leq C\|w(t_j)\|_{H_x^1}+\widetilde{C}\sum_{i=1}^4 (\|w\|_{X^1(I_j)}^{5-i}\|v\|^i_{Z'(I_j)}
+\|v\|_{X^1(I_j)}\|w\|_{X^1(I_j)}^{5-i}\|v\|_{Z'(I_j)}^{i-1})\nonumber\\
& \quad + \widetilde{C}\|w\|^5_{X^1(I_j)}+\widetilde{C}|J|^{\frac{1}{20}}\|v+w\|_{X^1(I_j)}\|v+w\|_{Z'(I_j)}^2\nonumber\\
& \leq \Big[C((2C)^{j-1}|J|^{\frac{1}{20}})\Big]+\Big[\widetilde{C}\sum_{i=1}^3 ((2C)^j|J|^{\frac{1}{20}})^{5-i}\eta^{i}
+\widetilde{C}\|v\|_{X^1(I)}\sum_{i=2}^3((2C)^j|J|^{\frac{1}{20}})^{5-i}\eta^{i-1}\nonumber\\
& \quad +
\widetilde{C}\|v\|_{X^1(I)}((2C)^j|J|^{\frac{1}{20}})^{4}+\widetilde{C}|J|^{\frac{1}{20}}((2C)^j|J|^{\frac{1}{20}})\eta^2\nonumber\\
& \quad + \widetilde{C}|J|^{\frac{1}{20}}\|v\|_{X^1(I)}((2C)^j|J|^{\frac{1}{20}})^2+
\widetilde{C}|J|^{\frac{1}{20}}((2C)^j|J|^{\frac{1}{20}})^3\Big]\nonumber\\
&\quad + \Big[\widetilde{C}|J|^{\frac{1}{20}}\|v\|_{X^1(I)}\eta^2+\widetilde{C}((2C)^j|J|^{\frac{1}{20}})\eta^{4}
+\widetilde{C}\|v\|_{X^1(I)}((2C)^j|J|^{\frac{1}{20}})\eta^{3}\Big]\nonumber\\
& =:A_1+A_2+A_3
\end{align}
for some $\widetilde{C}>0$. We have $A_1=\frac{1}{2}(2C)^j |J|^{\frac{1}{20}}$. By choosing $\eta=\eta(\|v\|_{X^1(I)})=\eta(\|u(a)\|_{H_x^1})$ sufficiently small depending on $\|u(a)\|_{H_x^1}$ we have $A_3\leq
\frac{1}{4}(2C)^j |J|^{\frac{1}{20}}$. For $A_2$, we may choose $|J|\leq \widetilde{\eta}$ with $\widetilde{\eta}$ depending on $0\leq j\leq m$ so that $A_2\leq
\frac{1}{4}(2C)^j |J|^{\frac{1}{20}}$ is valid for all $j$. Indeed, the dependence of $J$ on $j$ can be expressed as on $\|u(a)\|_{H_x^1}$ via $j\leq m\leq C(\|v\|_{Z'(I)})= C(\|u(a)\|_{H_x^1})$, where the last equality is deduced from Theorem \ref{thm:black box}. Similarly we are able to show that by shrinking $\eta$ and $\widetilde{\eta}$ if necessary, we have
$$\|\Gamma_j(u_1)-\Gamma_j(u_2)\|_{X^1(I_j)}\leq \frac{1}{2}\|u_1-u_2\|_{X^1(I_j)}$$
for all $0\leq j\leq m-1$. The proof is analogous and we hence omit the details here. The claim then follows from the Banach fixed point theorem.

Now we close our proof by removing the smallness of $|J|$. By Lemma \ref{zhang uniform kinetic} we have $\|u\|_{L_t^\infty H_x^1(I)}<\infty$. Thus we may choose $(\eta,\widetilde{\eta})=(\eta,\widetilde{\eta})(\|u\|_{L_t^\infty H_x^1(I)})$ in a way such that the previous proof is valid for all $J=[a,b]$ for any $a\in I$ with $|J|\leq \widetilde{\eta}$. We now partition $I$ into disjoint consecutive subintervals $I=\cup_{j=0}^n\, J_j$ with $|J_j|\leq \widetilde{\eta}$ for all $0\leq j\leq n$, and the proof follows by applying the previous step to each $J_j$ and summing up.
\end{proof}




\bibliographystyle{acm}
\bibliography{BGCQNLS}

\end{document}